\documentclass[12pt]{amsart}

\usepackage{amssymb,latexsym}
\usepackage{verbatim,euscript,array}
\usepackage{fullpage,color}

\usepackage[colorlinks=true,linkcolor=blue,urlcolor=my_color,citecolor=magenta]{hyperref}

\definecolor{my_color}{rgb}{0,0.5,0.5}
\definecolor{MIXT}{rgb}{0.4,0.3,0.6}

\input {cyracc.def}
\numberwithin{equation}{section}

\font\tencyr=wncyr10 
\def\rus{\tencyr\cyracc}

\newtheorem{thm}{Theorem}[section]
\newtheorem{lm}[thm]{Lemma}
\newtheorem{prop}[thm]{Proposition}

\theoremstyle{remark}
\newtheorem{rmk}[thm]{Remark}

\theoremstyle{definition}
\newtheorem{ex}[thm]{Example}
\newtheorem{df}{Definition}
\newtheorem*{rema}{Remark}

\newenvironment{proof*}
{\noindent {\sl Proof.}\quad }{\hfill $\square$}

\newcommand {\ah}{{\mathfrak a}}
\newcommand {\be}{{\mathfrak b}}
\newcommand {\ce}{{\mathfrak c}}

\newcommand {\g}{{\mathfrak g}}

\newcommand {\gH}{{\eus H}}

\newcommand {\te}{{\mathfrak t}}
\newcommand {\ut}{{\mathfrak u}}

\newcommand {\sono}{{\mathfrak{so}}_{2n+1}}
\newcommand {\sone}{{\mathfrak{so}}_{2n}}

\newcommand {\gG}{{\eus A}}

\newcommand {\esi}{\varepsilon}
\newcommand {\ap}{\alpha}

\newcommand {\HW}{\widehat W}

\newcommand {\BC}{{\mathbb C}}

\newcommand {\hot}{{\mathsf{ht}}}

\newcommand {\rk}{{\mathsf{rk\,}}}
\newcommand {\rt}{{\mathsf{rt}}}

\newcommand {\supp}{{\mathsf{supp}}}

\newcommand {\GR}[2]{{\textrm{{\bf #1}}}_{#2}}

\newcommand {\Ab}{\mathfrak{Ab}}
\newcommand {\Abo}{\mathfrak{Ab}^o}

\newcommand {\beq}{\begin{equation}}
\newcommand {\eeq}{\end{equation}}

\newcommand{\curge}{\succcurlyeq}
\newcommand{\curle}{\preccurlyeq}
\renewcommand{\le}{\leqslant}
\renewcommand{\ge}{\geqslant}

\newenvironment{E6}[6]{
{\footnotesize\begin{tabular}{@{}c@{}}
{#1}-{#2}-\lower3.5ex\vbox{\hbox{{#3}\rule{0ex}{2.7ex}}
\hbox{\hspace{0.3ex}\rule{.2ex}{.8ex}\rule{0ex}{1.4ex}}\hbox{{#6}\strut}}-{#4}-{#5}
\end{tabular}}}

\newenvironment{E7}[7]{
{\footnotesize\begin{tabular}{@{}c@{}}
{#1}-{#2}-{#3}-\lower3.5ex\vbox{\hbox{{#4}\rule{0ex}{2.7ex}}
\hbox{\hspace{0.3ex}\rule{.2ex}{.8ex}\rule{0ex}{1.3ex}}\hbox{{#7}\strut}}-{#5}-{#6}
\end{tabular}}}

\newenvironment{E8}[8]{%
{\footnotesize\begin{tabular}{@{}c@{}}
{#1}-{#2}-{#3}-{#4}-\lower3.5ex\vbox{\hbox{{#5}\rule{0ex}{2.6ex}}
\hbox{\hspace{0.4ex}\rule{.2ex}{.8ex}\rule{0ex}{1.3ex}}\hbox{{#8}\strut}}-{#6}-{#7}
\end{tabular}}}

\newcommand{\eus}{\EuScript}

\begin{document}
\setlength{\parskip}{2pt plus 4pt minus 0pt}
\hfill {\scriptsize November 9, 2017} 
\vskip1.5ex

\title[Abelian ideals and amazing roots]{Abelian ideals and amazing roots}
\author{Dmitri I. Panyushev}
\address[]{Institute for Information Transmission Problems of the R.A.S., Bolshoi Karetnyi per. 19,  
127051 Moscow,  Russia}
\email{panyushev@iitp.ru}
\thanks{This research was carried out at the IITP RAS at the expense of the Russian Foundation for Sciences (project {\rus N0} 14-50-00150).}
\begin{abstract}

Let $\g$ be a simple Lie algebra with a Borel subalgebra $\be$. To any long positive 
root $\gamma$, one associates two ideals of $\be$: the abelian ideal $I(\gamma)_{\max}$ and
not necessarily abelian ideal $I\langle{\curge}\gamma\rangle$. It is known that 
$I(\gamma)_{\max} \subset I\langle{\curge}\gamma\rangle$,  and $\gamma$ is said to be amazing if
the equality holds. The set of amazing roots, $\gG$, is closed under the operation `$\vee$' in $\Delta^+$, 
and $\gamma\in\gG$ is said to be primitive, if it cannot be written as $\gamma_1\vee\gamma_2$
with incomparable amazing roots $\gamma_1,\gamma_2$. We classify the amazing roots and notice that 
the number of primitive roots equals $\rk\g$. Moreover, if $\Pi$ (resp. $\gG_{\sf pr}$)
is the set of simple (resp. primitive) roots, then there is a natural bijection $\Pi\longleftrightarrow
\gG_{\sf pr}$. We also study the subset $\gG_\gH$  of  $\gG$.
\end{abstract}
\maketitle

\section*{Introduction}

\noindent
Let $\g$ be a simple Lie algebra over $\BC$, with a  triangular decomposition 
$\g=\ut\oplus\te\oplus \ut^-$, where $\te$ is a Cartan subalgebra and $\be=\ut\oplus\te$ is a fixed 
Borel subalgebra.  
The general theory of abelian ideals of $\be$ is built on their relationship with the so-called 
{\it minuscule elements\/} of the affine Weyl group $\HW$, which is due to D.~Peterson (see Kostant's 
account in~\cite{ko98}).  The subsequent development has lead to spectacular results of 
combinatorial and representation-theoretic nature, see e.g.~\cite{cp1,cp3,ko04,imrn,jems,som05,suter}. 
In this note, we consider some combinatorial consequences.

Let $\Delta$ be the root system of $(\g,\te)$, $\Delta^+$ the set of positive roots corresponding to $\ut$,  
$\Pi$ the set of simple roots in $\Delta^+$, and $\theta$  the {highest root} in  $\Delta^+$. Then $W$ is the Weyl group and $\g_\gamma$ is the root space for $\gamma\in\Delta$.
We write $\Ab$ for the set of all abelian ideals of $\be$ and 
think of $\Ab$ as poset with respect to inclusion. 
If $\ah\in\Ab$, then $\ah\subset \ut$, and thereby it is a sum of certain root spaces. Therefore,
we identify $\ah$ with the corresponding subset $I$ of $\Delta^+$.

Let $\Abo$ denote the set of nonzero abelian ideals and $\Delta^+_l$  the set
of long positive roots.  In the {\sf A-D-E} case, all roots are assumed to be long.
In~\cite[Sect.\,2]{imrn}, we defined a surjective mapping $\tau: \Abo \to \Delta^+_l$ and studied its fibres. 
If  $\tau(\ah)=\mu$, then $\mu\in\Delta^+_l$ is called the {\it rootlet\/} of $\ah$, also denoted $\rt(\ah)$ or 
$\rt(I)$.  
Letting $\Ab_\mu=\tau^{-1}(\mu)$, we get a  partition of $\Abo$ parameterised by
$\Delta^+_l$. Each $\Ab_\mu$ is  a subposet of $\Ab$ and, moreover, 
$\Ab_\mu$ has a unique minimal and unique maximal element~\cite[Sect.\,3]{imrn}.  
The corresponding subsets of $\Delta^+$ are denoted by
$I(\mu)_{\min}$ and $I(\mu)_{\max}$. 
We equip  $\Delta^+$  with the usual partial ordering `$\curge$'. It is proved in~\cite[Appendix]{jems} that,
for any $\mu_1,\mu_2\in \Delta^+$, there is a unique minimal root $\mu$ such that 
$\mu\curge \mu_i$, $i=1,2$. This root $\mu$ is denoted by $\mu_1\vee \mu_2$.
 \\ \indent 
Set $I\langle{\curge}\mu\rangle=\{\gamma\in\Delta^+\mid \gamma\curge\mu\}$. 
Then $I(\mu)_{\max}\subset I\langle{\curge}\mu\rangle$ for any 
$\mu\in\Delta^+_l$, see~\cite[Theorem\,3.5]{jems}.
In this note, we study the (long) positive roots having the property that
$I(\mu)_{\max}= I\langle{\curge}\mu\rangle$. Such roots are said to be {\it amazing}. Let 
$\gG=\gG(\Delta)$ be the set of all amazing roots. We prove that $\gG\subset\Delta^+$ is closed under  
`$\vee$'. Let us say that $\gamma\in\gG$ is {\it primitive}, if it cannot be obtained as $\gamma_1\vee\gamma_2$, where $\gamma_1$ and $\gamma_2$ are incomparable {\bf amazing} roots.
(`{\it Incomparable}' means that
$\gamma_1\not\curle \gamma_2$ and $\gamma_2\not\curle \gamma_1$.) 
Write $\gG_\mathsf{pr}$ 
for the set of all primitive roots. We obtain the 
classification of the amazing roots (Theorem~\ref{thm:classif}) and then make some surprising 
observations:

\textbullet\quad  
$\#\gG_\mathsf{pr}=\#\Pi$ and there is a natural bijection
$\gG_\mathsf{pr}\longleftrightarrow \Pi$ (Theorem~\ref{thm:bij1}). 
\\ Unfortunately, I have no conceptual explanation, but I believe there must be one.
\\ \indent 
\textbullet\quad  
Consider $\gH=\{\gamma\in \Delta^+ \mid (\gamma, \theta)\ne 0\}$ and $\gG_{\gH}=\gG\cap\gH$. 
Suppose that $\theta$ is {\bf fundamental}. We 
prove that  $\#\gG_{\gH}=\# \Pi_l+1$,  $\gG_{\gH}$ is a subposet of $\Delta^+$, and the Hasse diagram 
of $\gG_{\gH}$ is a tree that is isomorphic to the {\bf extended} Dynkin diagram with deleted short simple roots. (In particular, for
$\GR{D}{n}$, and $\GR{E}{n}$, one obtains the full extended Dynkin diagram.) There is also
a natural bijection between $\Pi_l$ and the edges of $\gG_{\gH}$ (Theorem~\ref{thm:bij2}).
\\ \indent 
\textbullet\quad   
For $\Delta$ of type $\GR{B}{n},\GR{F}{4}$, and $\GR{G}{2}$,  one can 
add to $\gG_\gH$ some {\bf short} roots from $\gH$ so that the larger subposet,
$\widetilde{\gG_\gH}$, will represent the full extended Dynkin diagram. This also yields a bijection between $\Pi$ and the edges of $\widetilde{\gG_\gH}$.
In all cases, the vertex $\{\theta\}$ in $\gG_\gH$  or $\widetilde{\gG_\gH}$ corresponds to the extra node in the extended Dynkin 
diagram.

\section{Preliminaries}    \label{sect:1}

\noindent
The partial ordering `$\curle$' in $\Delta^+$ is defined by the rule  that $\mu\curle\nu$ if 
$\nu-\mu$ is a non-negative integral linear combination of simple roots. 
Any $\be$-stable subspace $\ce\subset \ut$ is a sum of certain root spaces in $\ut$,  i.e.,
$\ce=\bigoplus_{\gamma\in I_\ce}\g_\gamma$.  The relation $[\be,\ce]\subset \ce$ is equivalent to
that $I=I_\ce$ is an {\it upper ideal\/} of the poset $(\Delta^+, \curle)$, i.e., 
if $\nu\in I$, $\gamma\in\Delta^+$, and $\nu\curle \gamma$, then $\gamma\in I$.
We only work in the combinatorial setting, so that a $\be$-ideal $\ce\subset\ut$
is being identified with the corresponding upper ideal $I$ of $\Delta^+$. 
The property of being abelian additionally means that
$\gamma'+\gamma''\not\in \Delta^+$ for all $\gamma',\gamma''\in I$.

Now, we have $\Pi=\{\ap_1,\dots,\ap_n\}$, the vector space $V=\oplus_{i=1}^n{\mathbb R}\ap_i$, 
the  Weyl group $W$ generated by  simple reflections $s_1,\dots,s_n$,  and a $W$-invariant inner 
product $(\ ,\ )$ on $V$. If $\mu=\sum_{i=1}^n c_i\ap_i\in\Delta^+$, then $\hot(\mu):=\sum c_i$ and 
$[\mu:\ap_i]:=c_i$. The support of $\mu$ is $\supp(\mu)=\{\ap\in\Pi\mid [\mu:\ap]\ne 0\}$. As is 
well known, $\supp(\mu)$ is a connected subset on the Dynkin diagram.
Given $\mu_1$ and $\mu_2$, the root 
$\mu_1\vee \mu_2$ is explicitly computed as follows (see~\cite{jems}):
\\ \indent
$(\diamond_1)$ \ If $(\ap,\beta)=0$ for all $\ap\in \supp(\mu_1)$ and $\beta\in\supp(\mu_2)$ (the case of disjoint supports), 
then $\mu_1\vee \mu_2=\mu_1+\mu_2+\sum_{\ap\in C}\ap$, where $C$ is the unique chain on the Dynkin diagram connecting both supports;
\\ \indent
$(\diamond_2)$ \ In all other cases, $[\mu_1\vee \mu_2:\ap]=\max\{[\mu_1:\ap],[\mu_2:\ap]\}$ for any $\ap\in\Pi$.

The mapping $\tau: \Abo\to \Delta^+_l$ is defined via the use of the minuscule element of $\widehat W$
corresponding to $I\in \Abo$. We refer to \cite[Sect.\,2]{imrn} for details.
By \cite[Prop.\,2.5]{imrn}, this map is surjective.

The Heisenberg set $\gH:=\{\gamma\in \Delta^+ \mid (\gamma, \theta)\ne 0\}$ plays a prominent role in the theory: 

\textbullet\quad  $\#\tau^{-1}(\mu)=1$ (i.e., $I(\mu)_{\min}=I(\mu)_{\max}$) if and only if
$\mu\in\gH$~\cite[Theorem\,5.1]{imrn}.

\textbullet\quad   $I=I(\mu)_{\min}$ for some $\mu$  if and only if $I\subset
     \gH$~\cite[Theorem\,4.3]{imrn};

\textbullet\quad  For $I\in\Abo$, we have $I\in\Ab_\mu$ if and only if $I\cap\gH=I(\mu)_{\min}\cap\gH$~\cite[Prop.\,3.2]{jems}.

\noindent
For any $\gamma\in\Delta^+$, set $I\langle{\curge}\gamma\rangle=\{\nu\in\Delta^+\mid
\nu\curge\gamma\}$. 
The upper ideal  $I\langle{\curge}\gamma\rangle$ is not necessarily abelian, and  $\gamma$ is said to be 
{\it commutative}, if it is abelian. Write $\Delta^+_{\sf com}$ for the set of all commutative roots.
This notion was introduced in \cite{jac06}, and the set  $\Delta^+_{\sf com}$ for any irreducible root
system is described in~\cite[Theorem\,4.4]{jac06}.

\section{Some properties of amazing roots}   \label{sect:2}

\begin{df}  \label{def:glor}
We say that $\gamma\in \Delta^+_l$ is {\it amazing}, if $I(\gamma)_{max}=I\langle{\curge}\gamma\rangle$.
The set of amazing roots is denoted by $\gG$.
\end{df}
It follows that $\gG\subset \Delta^+_{\sf com}$.
Our goal is to classify the amazing roots.  Since $\theta\in I(\theta)_{\min}\subset I(\theta)_{\max}
\subset I\langle{\curge}\theta\rangle=\{\theta\}$, we see that $\theta\in \gG$.
\begin{lm}    \label{lm:1}
If $\gamma_1,\gamma_2\in \gG$, then $\gamma_1\vee\gamma_2\in\gG$.
\end{lm}
\begin{proof}  In view of the definition and existence of `$\vee$', for all $\gamma_1,\gamma_2\in\Delta^+_l$, 
we have \\  \centerline{
$I\langle{\curge}\gamma_1\rangle\cap I\langle{\curge}\gamma_2\rangle=
I\langle{\curge}(\gamma_1\vee \gamma_2)\rangle$.}

On the other hand,  
$I(\gamma_1)_{\max}\cap I(\gamma_2)_{\max}=I(\gamma_1\vee \gamma_2)_{\max}$, 
see~\cite[Theorem\,2.5]{jems}.
\end{proof}

This provides some ground for the following 
\begin{df}  \label{def:primit}
An amazing root $\gamma$ is said to be {\it primitive}, if it cannot be written as $\gamma_1\vee\gamma_2$, where $\gamma_1$ and $\gamma_2$ are incomparable amazing roots. 
\end{df}

Let us say that $\gamma\in \Delta^+$ is {\it inconvenient}, if there is a unique $\ap\in\Pi$ such that
$\gamma-\ap\in \Delta^+\cup\{0\}$. Equivalently, $\gamma$ cannot be written as $\gamma_1\vee\gamma_2$, where $\gamma_1$ and $\gamma_2$ are incomparable positive roots. Clearly, the simple roots are inconvenient, but there are always non-simple inconvenient roots, if $\Delta$ is not of type $\GR{A}{n}$.

\noindent  It follows from these definitions that \\
\centerline{ (amazing) \& (inconvenient) $\Rightarrow$ (primitive) }

\noindent Moreover, our classification will show that if $\gamma$ is amazing, then ``primitive'' implies
``inconvenient''. That is, {\sl a posteriori}, the equivalence holds.

Consider the natural map $\varkappa: \Delta^+_{\sf com} \to \Delta^+_l$, $\gamma\mapsto 
\rt(I\langle{\curge}\gamma\rangle)$. Recall that if $I\in\Abo$ and $\mu=\rt(I)$, then  $I\subset 
I\langle{\curge}\mu\rangle$. Hence $\varkappa(\gamma)\curle\gamma$, and  $\gamma$ is amazing if and only if $\varkappa(\gamma)=\gamma$. Since $\rt(I)=\rt(I\cap\gH)$ for any $I\in\Abo$~\cite[Prop.\,3.2]{jems}, it suffices to know the rootlet for all abelian ideals of the form $I\langle{\curge}\mu\rangle\cap\gH$ with $\mu\in\Delta^+_{\sf com}$.

\begin{lm}    \label{lm:simple-glor}
If $\ap\in\Pi\cap \Delta^+_{\sf com}$, then $\ap\in \gG$. 
\end{lm}
\begin{proof}
For $\ap\in\Pi$, one has $\ap\in\Delta^+_{\sf com}$ if and only if $[\theta:\ap]=1$, and such an $\ap$ is 
necessarily long. As noted above, $\varkappa(\ap)\curle \ap$. Hence $\varkappa(\ap)=\ap$ and $\ap$
is amazing.
\end{proof}

In order to characterise the amazing roots in $\gH$ we need more notation and facts.
\\ \indent
\textbullet\quad If $\mu\in\Delta^+$, then a {\it path from $\mu$ to} $\theta$ is a sequence of roots
$(\mu_0=\mu, \mu_1,\dots, \mu_m=\theta)$ such that $\mu_{i+1}-\mu_i\in\Pi$. The length of the path is 
$m=\hot(\theta)-\hot(\mu)$. Clearly, 
$I\langle{\curge}\mu\rangle$ is the union of the roots occurring in all paths from $\mu$ to $\theta$. (Note that there is no restrictions on the length of $\mu$.) 
\\ \indent
\textbullet\quad For $\gamma\in\Delta^+_l$, there is a unique element of minimal length in $W$ that
takes $\theta$ to $\gamma$, denoted $w_\gamma$. Then 
$\ell(w_\gamma)=\hot(\theta^\vee)-\hot(\gamma^\vee)$, where $\theta^\vee$ and $\gamma^\vee$ 
are regarded as elements of the dual root system $\Delta^\vee$~\cite[Theorem\,4.1]{imrn}.
\\ \indent
\textbullet\quad 
For any $\gamma\in\Delta^+_l$, a {\it jump-path from $\gamma$ to} $\theta$ is a sequence of roots
$(\gamma_0=\gamma, \gamma_1,\dots, \gamma_t=\theta)$ such that 
$\gamma_{i+1}=s_{j_i}(\gamma_{i})$ for some
$\ap_{j_i}\in\Pi$ and $(\ap_{j_i}, \gamma_i)<0$. In this case $s_{j_1}\cdots s_{j_t}(\theta)=\gamma$ and 
$w_\gamma=s_{j_1}\cdots s_{j_t}$ is a reduced decomposition~\cite[Section\,4]{imrn}. 
The length of the jump-path, $t$, is equal to $\ell(w_\gamma)$.
Note that if $\ap_{j_i}$ is short, then $\gamma_{i+1}- \gamma_i$ is a non-trivial multiple of $\ap_{j_i}$, 
i.e., there are also some short roots between $\gamma_i$ and $\gamma_{i+1}$. Therefore, a jump-path 
is a path if and only if all the roots $\ap_{j_1}, \dots,\ap_{j_t}$ are long.

\begin{prop}   \label{prop:gamma-in-H}
For $\gamma\in\gH\cap \Delta^+_l$, the following conditions are equivalent:
\begin{itemize}
\item[\sf (i)] \ $\gamma$ is amazing;
\item[\sf (ii)] \ there is a unique path in $\Delta^+$ from $\gamma$ to $\theta$, and all roots in this path are long;
\item[\sf (iii)] \  $w_\gamma$ has a unique reduced decomposition that includes only long simple reflections.
\end{itemize}
\end{prop}
\begin{proof}
Since $\gamma\in \gH$, one has $I(\gamma)_{\min}=I(\gamma)_{\max}$~\cite[Theorem\,5.1]{imrn}. 

{\sf (i)$\Rightarrow$(ii)} \ It is known that $\# I(\gamma)_{\min}=\hot(\theta^\vee)-\hot(\gamma^\vee)+1$~\cite[Theorem\,4.2]{imrn} and any jump-path from $\gamma$ to $\theta$ contains this number of long roots.
All these roots (and all short roots between them, if any) belong to $I\langle{\curge}\gamma\rangle$.
Therefore, the condition $I(\gamma)_{\min}=I\langle{\curge}\gamma\rangle$ implies
that this jump-path is a path and it is unique.

{\sf (ii)$\Leftrightarrow$(iii)} \ This readily follows from the fact that any reduced decomposition of 
$w_\gamma$ yields a jump-path from $\gamma$ to $\theta$, and conversely, any jump-path determines
a reduced decomposition of $w_\gamma$~\cite[Section\,4]{imrn}.

{\sf (ii)$\Rightarrow$(i)} \ If $\gamma$ is not amazing, then $\#I\langle{\curge}\gamma\rangle>
\#I(\gamma)_{\min}$. This can only happen if there are several jump-paths from $\gamma$ to $\theta$, 
or the unique jump-path includes a ``jump" over short roots; that is, {\sf (ii)} does not hold.
\end{proof}

\begin{lm}   \label{lm:elem-extens}
Suppose that $\gamma\in \gG$ and  $J=I\langle{\curge}\gamma\rangle\cup\{\nu\}$ is an abelian ideal for
some $\nu\in\Delta^+$. Then $\nu\in \gH$.
\end{lm}
\begin{proof}
Since $I\langle{\curge}\gamma\rangle=I(\gamma)_{\max}$, one has $\rt(J)\ne \rt(I\langle{\curge}\gamma\rangle)$. (This actually means that $\rt(J)\prec \rt(I\langle{\curge}\gamma\rangle)$. 
Hence $J\cap\gH\ne I\langle{\curge}\gamma\rangle\cap\gH$, see \cite[Prop.\,3.2]{jems}, i.e., 
$\nu\in \gH$.
\end{proof}

\begin{prop}    \label{prop:polnaya-xar}
For any $\gamma\in\Delta^+_{\sf com}$, we have 
\beq   \label{eq:polnaya-xar}
  \hot(\theta^\vee)-\hot(\gamma^\vee)+1  \le \# (I\langle{\curge}\gamma\rangle\cap\gH) .
\eeq
Furthermore, $\gamma$ is amazing if and only if the equality holds in Eq.~\eqref{eq:polnaya-xar}.
\end{prop}
\begin{proof}
Since $I(\gamma)_{\max}\subset I\langle{\curge}\gamma\rangle$, we have
\[
  I(\gamma)_{\min}=I(\gamma)_{\max}\cap\gH\subset I\langle{\curge}\gamma\rangle\cap \gH .
\]
Because $\#I(\gamma)_{\min}=\hot(\theta^\vee)-\hot(\gamma^\vee)+1$,  we see that 
Eq.~\eqref{eq:polnaya-xar} holds.
Furthermore, $I(\gamma)_{\max}\cap\gH=I\langle{\curge}\gamma\rangle\cap \gH$ if and only if 
$I(\gamma)_{\max}=I\langle{\curge}\gamma\rangle$, i.e., $\gamma$ is amazing.
\end{proof}

\begin{ex}   \label{ex:non-glor-Dn} 
As usual, the simple roots of $\Delta(\GR{D}{n})$, $n\ge 4$, are $\ap_i=\esi_i-\esi_{i+1}$ ($i=1,\dots,n-1$)
and $\ap_n=\esi_{n-1}+\esi_n$. Then $\theta=\esi_1+\esi_2$.
\\ \indent
{\sf (1)} \ For $\gamma=\esi_2+\esi_j$ ($j\ge 3$), we have 
$\gamma\in \Delta^+_{\sf com}\cap\gH$ and $I\langle{\curge}\gamma\rangle=\{
\esi_2+\esi_k\ (3\le k\le j), \ \esi_1+\esi_l\ (2\le l\le j)\}$. However, if $j\ge 4$, then
$I\langle{\curge}\gamma\rangle\cup \{\esi_3+\esi_4\}$ is an abelian ideal and 
$\esi_3+\esi_4\not\in\gH$. Hence $\esi_2+\esi_j$ is not amazing for $j\ge 4$.
\\ \indent
{\sf (2)} \ If $\gamma=\esi_i+\esi_{i+1}$ ($i=1,\dots,n-2$), then $\hot(\gamma^\vee)=\hot(\gamma)=2n-2i-1$ and $I\langle{\curge}\gamma\rangle=\{\esi_k+\esi_l\mid k<l\le i+1\}$. Here
$I\langle{\curge}\gamma\rangle\cap\gH=\{\esi_k+\esi_l\mid k<l\le i+1\ \& \ k\le 2\}$. Hence
$
   \#\bigl(I\langle{\curge}\gamma\rangle\cap\gH\bigr)=2i-1=\hot(\theta^\vee)-\hot(\gamma^\vee)+1
$
and $\gamma\in\gG$.
\end{ex}

\begin{prop}    \label{prop:highest-in-dopoln}
Let $\tilde\theta$ be a highest root in the (possibly reducible) subsystem $\Delta^+\cap\theta^\perp=
\Delta^+\setminus \gH$. If $\tilde\theta$ is long, then it is amazing.
\end{prop}
\begin{proof}
If $\ap\in \Pi\cap\theta^\perp$, then $(\ap,\tilde\theta)\ge 0$ and $\ap+\tilde\theta\not\in\Delta$. While if 
$\ap\in \Pi\cap\gH$ and $(\ap,\tilde\theta)<0$, then $\ap+\tilde\theta\in\gH$. This means that
$I\langle{\curge}\tilde\theta\rangle \setminus\{\tilde\theta\}\subset \gH$.
Assuming that $\tilde\theta\in\Delta^+_l$, consider the poset $\Ab_{\tilde\theta}$.
\\ \indent
\textbullet\quad If $\Pi\cap\gH=\{\tilde\ap\}$ (i.e., $\Delta$ is not of type $\GR{A}{n}$), then one has $(\tilde\ap,\tilde\theta)<0$. It then follows from~\cite[Prop.\,5.3]{imrn} that $\#\Ab_{\tilde\theta}=2$. In other words,
$\Ab_{\tilde\theta}=\{I(\tilde\theta)_{\min}, I(\tilde\theta)_{\max}\}$ and $\# I(\tilde\theta)_{\max}=
\# I(\tilde\theta)_{\max}+1$. Furthermore, $I(\tilde\theta)_{\max}\subset I\langle{\curge}\tilde\theta\rangle$
and $I(\tilde\theta)_{\max}\not\subset\gH$. The only possibility is 
$I(\tilde\theta)_{\max}= I\langle{\curge}\tilde\theta\rangle$ (and hence $I(\tilde\theta)_{\min}=
I\langle{\curge}\tilde\theta\rangle \setminus\{\tilde\theta\}$). Thus, $\tilde\theta$ is amazing.
\\ \indent
\textbullet\quad For $\Delta$ of type $\GR{A}{n}$, we have $\Pi\cap\gH=\{\ap_1,\ap_n\}$. However,
$(\tilde\theta,\ap_1)=(\tilde\theta,\ap_n)<0$, and we still can conclude that $\#\Ab_{\tilde\theta}=2$.
The rest is the same.
\end{proof}

\begin{rema}
It can happen that a highest root in $\Delta^+\cap\theta^\perp$ is short, see $\GR{B}{3}$ or $\GR{G}{2}$.
For $\GR{D}{n}$ ($n\ge 5$) or $\GR{B}{n}$ ($n\ge 4$),
there are two long highest roots in $\Delta^+\cap\theta^\perp$.
\end{rema}

\section{Main results}   \label{sect:3}

\noindent
Using the above properties of amazing roots, one can obtain the full list of them in every root systems.
The numbering of simple roots follows \cite[Table\,1]{t41}. If $\gamma=\sum_{i=1}^n c_i\ap_i\in \Delta^+$, 
then we write $\gamma=[c_1\dots c_n]$. In the classical algebras, we also use the standard $\{\esi_i\}$ notation.  

\begin{thm}   \label{thm:classif}
The amazing roots in the irreducible root systems are:
 \begin{itemize}
\item[\sf (1)] \ $\gG(\GR{A}{n})=\Delta^+=\{ \esi_i-\esi_j\mid 1\le i < j\le n+1\}$; 
\item[\sf (2)] \ $\gG(\GR{B}{n})=\{\esi_1-\esi_2;\ \esi_1+\esi_j \, (j=2,\dots,n);\ \esi_i+\esi_{i+1} \,
(i=2,\dots,n-1)  \}$; 
\item[\sf (3)] \ $\gG(\GR{C}{n})=\Delta^+_l=\{2\esi_1,\dots,2\esi_n\}$; 
\item[\sf (4)] \ $\gG(\GR{D}{n})=\{\esi_1-\esi_2, \esi_1-\esi_n,\,\esi_1+\esi_j \, (j=2,\dots,n),\, 
\esi_i+\esi_{i+1}\, (i=2,\dots,n-1), \esi_{n-1}-\esi_n  \}$; 
\item[\sf (5)] \ $\gG(\GR{E}{6})=\Bigl\{\boldsymbol{\ap_1},\boldsymbol{\ap_5}, 
\emph{\begin{E6}{1}{1}{1}{1}{1}{0}\end{E6}},
\emph{\bf \begin{E6}{1}{2}{2}{1}{0}{1}\end{E6}}, 
\emph{\bf \begin{E6}{0}{1}{2}{2}{1}{1}\end{E6}},
\emph{\begin{E6}{1}{1}{2}{2}{1}{1}\end{E6}}, 
\emph{\begin{E6}{1}{2}{2}{1}{1}{1}\end{E6}},
\emph{\begin{E6}{1}{2}{2}{2}{1}{1}\end{E6}}, \\  \phantom{\text{\begin{E7}{1}{1}{2}{2}{2}{1}{1}\end{E7}} }
\emph{\bf \begin{E6}{1}{2}{3}{2}{1}{1}\end{E6}}, 
\emph{\bf \begin{E6}{1}{2}{3}{2}{1}{2}\end{E6}}=\theta \Bigr\}$;

\item[\sf (6)] \ $\gG(\GR{E}{7})=\Bigl\{\boldsymbol{\ap_1},   
\emph{\bf \begin{E7}{1}{2}{2}{2}{1}{0}{1}\end{E7}},
\emph{\bf \begin{E7}{0}{1}{2}{3}{2}{1}{2}\end{E7}},
\emph{\begin{E7}{1}{1}{2}{3}{2}{1}{2}\end{E7}},
\emph{\bf \begin{E7}{1}{2}{3}{3}{2}{1}{1}\end{E7}},
\emph{\begin{E7}{1}{2}{2}{3}{2}{1}{2}\end{E7}}, 
\emph{\begin{E7}{1}{2}{3}{3}{2}{1}{2}\end{E7}}, \\ \phantom{\text{\begin{E7}{1}{1}{2}{2}{2}{1}{1}\end{E7}} }
\emph{\bf \begin{E7}{1}{2}{3}{4}{2}{1}{2}\end{E7}},
\emph{\bf \begin{E7}{1}{2}{3}{4}{3}{1}{2}\end{E7}},
\emph{\bf \begin{E7}{1}{2}{3}{4}{3}{2}{2}\end{E7}}=\theta
 \Bigr\}$;

\item[\sf (7)] \ $\gG(\GR{E}{8})=\Bigl\{
\emph{\bf \begin{E8}{0}{1}{2}{3}{4}{3}{2}{2}\end{E8}},
\emph{\bf \begin{E8}{1}{2}{3}{4}{5}{3}{1}{3}\end{E8}},
\emph{\bf \begin{E8}{1}{2}{3}{4}{5}{4}{2}{2}\end{E8}},
\emph{\begin{E8}{1}{2}{3}{4}{5}{3}{2}{3}\end{E8}},
\emph{\begin{E8}{1}{2}{3}{4}{5}{4}{2}{3}\end{E8}},
\\ \phantom{\text{\begin{E8}{0}{1}{1}{2}{2}{2}{1}{1}\end{E8}} }
\emph{\bf \begin{E8}{1}{2}{3}{4}{6}{4}{2}{3}\end{E8}},
\emph{\bf \begin{E8}{1}{2}{3}{5}{6}{4}{2}{3}\end{E8}},
\emph{\bf \begin{E8}{1}{2}{4}{5}{6}{4}{2}{3}\end{E8}},
\emph{\bf \begin{E8}{1}{3}{4}{5}{6}{4}{2}{3}\end{E8}},
\emph{\bf \begin{E8}{2}{3}{4}{5}{6}{4}{2}{3}\end{E8}}=\theta
 \Bigr\}$;
\item[\sf (8)] \ $\gG(\GR{F}{4})=\{\emph{\text{\small [2210], [2421], [2431],[2432]}}=\theta\}$;
\item[\sf (9)] \ $\gG(\GR{G}{2})=\{ \emph{\text{\small [31], [32]}}=\theta\}$.
\end{itemize}
\end{thm}
\begin{proof}
The $\GR{A}{n}$-case  is simple, because here $\Pi\cap \Delta^+_{\sf com}=\Pi\subset\gG$. Hence
$\Delta^+\subset\gG$ in view of Lemma~\ref{lm:1}.  For $\GR{C}{n}$, there is a few long roots, and a straightforward calculation is
easy. For all other types, we first determine $\gG_\gH$ using Proposition~\ref{prop:gamma-in-H}.
Then we use Lemma~\ref{lm:elem-extens} as a necessary condition for $\gamma\in \Delta^+_{\sf com}\setminus\gH$ to be amazing. This appears to be sufficient for detecting all non-amazing roots in $\GR{B}{n}$ and $\GR{D}{n}$. We also use Proposition~\ref{prop:polnaya-xar}
for a quick verification that a root $\gamma$ is amazing, cf. Example~\ref{ex:non-glor-Dn}(2).
\end{proof}

Using our classification, we make a surprising observation.
\begin{thm}   \label{thm:bij1}
For any irreducible root system $\Delta$, 
\begin{itemize}
\item[\sf (i)] \ the number of primitive roots equals $\#\Pi$;
\item[\sf (ii)] \ there is a natural bijection $\gG_\mathsf{pr}\stackrel{1:1}{\longrightarrow} \Pi$. This bijection 
takes $\gamma\in \gG_\mathsf{pr}$ to the unique $\ap\in\Pi$ such that 
$\gamma-\ap\in \Delta^+\cup\{0\}$.
\end{itemize}
\end{thm}
\begin{proof}
(i) Using the explicit formula for `$\vee$',
it is not hard to single out the primitive roots from the lists in Theorem~\ref{thm:classif}. In particular,
$\gG_{\sf pr}=\gG$ for $\GR{C}{n}, \GR{F}{4}, \GR{G}{2}$ and 
$\gG_{\sf pr}=\Pi$ for $\GR{A}{n}$.

For $\GR{D}{n}$, we obtain $\gG_{\sf pr}=\{\ap_1, \ap_{n-1}, \ap_n,
\esi_i+\esi_{i+1}\ (i=1,\dots,n-3)\}$. 
\\
For instance, set $\gamma_i:=\esi_i+\esi_{i+1}=
\ap_i+2(\ap_{i+1}+\dots+\ap_{n-2})+ \ap_{n-1}+ \ap_{n}$ $(i=1,\dots,n-3)$.  If $i\ge 2$, then $\ap_1$ and 
$\gamma_i$ are incomparable  and hence 
\[
\text{$\ap_1\vee\gamma_i=\ap_1+\dots+\ap_i+2(\ap_{i+1}+\dots+\ap_{n-2})+ \ap_{n-1}+ \ap_{n}=\esi_1+\esi_{i+1}\not\in\gG_{\sf pr}$ for $i\ge 2$.} 
\]
Also $\ap_1\vee\ap_{n-1}=\ap_1+\dots+\ap_{n-1}=\esi_1-\esi_n$, 
$\ap_1\vee\ap_{n}=\ap_1+\dots+\ap_{n-2}+\ap_n=\esi_1+\esi_n$, and $\ap_{n-1}\vee\ap_n=
\ap_{n-2}+\ap_{n-1}+\ap_n=\esi_{n-2}+\esi_{n-1}$. This discards all non-primitive roots.

For $\GR{B}{n}$, we similarly obtain $\gG_{\sf pr}=\{\ap_1=\esi_1-\esi_2, \esi_i+\esi_{i+1}\ (i=1,\dots,n-1)\}$.

For $\GR{E}{n}$, the primitive roots are presented in {\bf bold} in Theorem~\ref{thm:classif}. 

\noindent
(ii) One readily verifies that the simple roots that can be subtracted 
from different primitive roots are also different. For instance, for $\GR{D}{n}$, the only simple root that can be subtracted from $\gamma_i$ is $\ap_{i+1}$ $(i=1,\dots,n-3)$.
\end{proof}

\begin{rema} If $\Delta$ is not of type $\GR{A}{n}$, $n\ge 2$, then $\theta$ is primitive.
\end{rema}
Recall that $\gG_\gH=\gG\cap\gH$.

\begin{thm}   \label{thm:bij2}
If $\theta$ is fundamental (i.e., $\Delta$ is not of type $\GR{A}{n}$ or $\GR{C}{n}$), then
\begin{itemize}
\item[\sf (i)] \ $\#\gG_\gH=\#\Pi_l+1$; 
\item[\sf (ii)] \ $\gG_\gH$ is a subposet of $\Delta^+$ and its Hasse diagram is a tree that is isomorphic 
to the sub-diagram of
the extended Dynkin diagram, where the nodes corresponding to the short roots are omitted;
there is a natural bijection between the edges of the tree $\gG_\gH$ and $\Pi_l$.
 \item[\sf (iii)] \ in the non-simply laced cases $\GR{B}{n}$,  $\GR{F}{4}$, and  $\GR{G}{2}$, 
one can add to $\gG_\gH$ some \emph{\/short commutative} roots so that the resulting subposet, 
say $\widetilde{\gG_\gH}$, represents the whole extended Dynkin diagram; there is a natural bijection 
between the edges of the tree $\widetilde{\gG_\gH}$ and $\Pi$.
 \item[\sf (iv)] \ In all cases, the vertex $\theta\in \gG_\gH$ represents the extra node of the extended Dynkin diagram.
\end{itemize}
\end{thm}
\begin{proof} 
Let $\tilde{\eus D}$ be the extended Dynkin diagram of $\Delta$, where the extra node is denoted by 
$\ap_0$. It is a tree, since type $\GR{A}{n}$ is excluded.
\\ \indent
{\sf (1)} \ In the simply-laced case, all assertions follow from the fact that there is a natural bijection 
between $\gG_\gH$ and the chains in $\tilde{\eus D}$ starting at $\ap_0$. 
\\ \indent
\textbullet \ \ Let $\ap_0,\ap_{i_1},\dots,\ap_{i_m}$ be the unique chain connecting $\ap_0$ and 
$\ap_{i_m}$ in $\tilde{\eus D}$. The corresponding root is defined to be $\gamma=\theta-\sum_{j=1}^m\ap_{i_j}$. Then $\gamma\in\gH$, because $(\theta,\ap_{i_j})=0$ for $j\ge 2$, and it is easily seen that
$\gamma, \gamma+\ap_{i_m}, \gamma+\ap_{i_m}+\ap_{i_{m-1}},\dots$ is the unique path from 
$\gamma$ to $\theta$. Hence $\gamma$ is amazing in view of Prop.~\ref{prop:gamma-in-H}.
\\ \indent
\textbullet \ \ Conversely, suppose that $\gamma\in\gG_\gH$ and let 
$\gamma=\gamma_0,\gamma_1,\dots,\gamma_m=\theta$ be the unique path from $\gamma$ to 
$\theta$, with $\gamma_i-\gamma_{i-1}=\ap_{j_i}\in\Pi$.
Then the uniqueness implies that $(\ap_{j_i},\ap_{j_{i+1}})<0$ and $\ap_0,\ap_{j_m},\dots,\ap_{j_1}$ is a
chain in $\tilde{\eus D}$.
Let us explain the first step. We have $\gamma_1=\gamma+\ap_{j_1}\in\Delta^+$, hence
$(\gamma, \ap_{j_1})<0$. Since $\gamma+\ap_{j_2}\not\in\Delta$ (uniqueness!), 
$(\gamma, \ap_{j_2})\ge 0$. Next, $\gamma_2=(\gamma+\ap_{j_1})+\ap_{j_2}\in\Delta^+$. 
Then $(\gamma+\ap_{j_1}, \ap_{j_2})<0$  and hence $(\ap_{j_1}, \ap_{j_2})<0$.
\\ \indent
{\sf (2)} \ The above argument can be carried over to $\gG_\gH$ in the {\sf B-F-G} case. 
As for $\widetilde{\gG_\gH}$, we resort to case-by-case verifications (cf. figures below).
\end{proof}

\noindent
It follows that, for $\GR{D}{n}$ and  $\GR{E}{n}$, one immediately obtains the whole extended Dynkin diagram. We provide in Fig.~\ref{fig:Dn} the respective configuration for $\GR{D}{n}$. The numbers on the edges indicate the corresponding simple roots.

\begin{figure}[htb]
\begin{center}
\setlength{\unitlength}{0.02in}
\begin{picture}(190,42)(-10,-10)

\put(-45,6){\small $\gG_\gH(\GR{D}{n})$:}
\multiput(30,8)(25,0){2}{\circle{5}}
\multiput(33,8)(25,0){2}{\line(1,0){19}}
\multiput(5,-5)(0,26){2}{\circle{5}}
\put(8,11){\oval(43,20)[rt]}
\put(8,5){\oval(43,20)[rb]}
\put(-7,27){\footnotesize $\esi_1{-}\esi_n$}
\put(1,6.5){\footnotesize $\esi_1{+}\esi_{n{-}1}$}
\put(-7,-15){\footnotesize $\esi_1{+}\esi_{n}$}
\put(50,15){\footnotesize $\esi_1{+}\esi_{n{-}2}$}

\put(82,6){$\cdots$}
\multiput(120,8)(25,0){2}{\circle{5}}
\multiput(98,8)(25,0){2}{\line(1,0){19}}

\multiput(120,-5)(50,26){2}{\circle{5}}
\put(167,11){\oval(43,20)[lt]}
\put(123,5){\oval(43,20)[rb]}

\put(108,15){\footnotesize $\esi_1{+}\esi_{4}$}
\put(168,27){\footnotesize $\esi_1{+}\esi_2=\theta$}
\put(150,6.5){\footnotesize $\esi_1{+}\esi_{3}$}
\put(118,-15){\footnotesize $\esi_2{+}\esi_{3}$}

{\color{blue}
\put(18,22){\footnotesize $n$}
\put(18,-11){\footnotesize $n{-}1$}
\put(36,2){\footnotesize $n{-}2$}
\put(131,2){\footnotesize $3$}
\put(150,22){\footnotesize $2$}
\put(140,-11){\footnotesize $1$}
}
\end{picture}
\caption{The subposet $\gG_\gH$ in $\Delta^+(\sone)$}
\label{fig:Dn}
\end{center}
\end{figure}
 
\noindent
We leave it to the reader to determine the subposet $\gG_\gH$ in case of 
$\GR{E}{n}$. For $\GR{B}{n}$ and $\GR{F}{4}$, the extended Dynkin diagram $\widetilde{\gG_\gH}$ sitting inside the set of commutative roots is shown in
Fig.~\ref{fig:Bn} and ~\ref{fig:F4}, respectively. 

\begin{figure}[htb]
\begin{center}
\setlength{\unitlength}{0.02in}
\begin{picture}(190,40)(-10,-10)
\put(-45,6){\small $\widetilde{\gG_\gH}(\GR{B}{n})$:}
\put(5,8){\circle*{5}}
\multiput(30,8)(25,0){2}{\circle{5}}
\multiput(33,8)(25,0){2}{\line(1,0){19}}
\multiput(9.8,7.1)(0,2){2}{\line(1,0){17.5}}
\put(7,6){$<$}

\put(2,15){\footnotesize $\esi_1$}
\put(21,15){\footnotesize $\esi_1{+}\esi_{n}$}
\put(50,15){\footnotesize $\esi_1{+}\esi_{n{-}1}$}

\put(82,6){$\cdots$}
\multiput(120,8)(25,0){2}{\circle{5}}
\multiput(98,8)(25,0){2}{\line(1,0){19}}

\multiput(120,-5)(50,26){2}{\circle{5}}
\put(167,11){\oval(43,20)[lt]}
\put(123,5){\oval(43,20)[rb]}

\put(108,15){\footnotesize $\esi_1{+}\esi_{4}$}
\put(168,27){\footnotesize $\esi_1{+}\esi_2=\theta$}
\put(150,6.5){\footnotesize $\esi_1{+}\esi_{3}$}
\put(118,-15){\footnotesize $\esi_2{+}\esi_{3}$}

{\color{blue}
\put(16.5,2){\footnotesize $n$}
\put(36,2){\footnotesize $n{-}1$}
\put(131,2){\footnotesize $3$}
\put(150,22){\footnotesize $2$}
\put(140,-11){\footnotesize $1$}
}
\end{picture}
\caption{The subposets $\gG_\gH$ and $\widetilde{\gG_\gH}$ in $\Delta^+(\sono)$}
\label{fig:Bn}
\end{center}
\end{figure}

\begin{figure}[htb]
\begin{center}
\setlength{\unitlength}{0.02in}
\begin{picture}(190,20)(-20,0)
\put(-55,6){\small $\widetilde{\gG_\gH}(\GR{F}{4})$:}

\multiput(5,8)(25,0){2}{\circle*{5}}
\multiput(55,8)(25,0){3}{\circle{5}}
\multiput(58,8)(25,0){2}{\line(1,0){19}}
\put(8,8){\line(1,0){19}}
\multiput(34.8,7.1)(0,2){2}{\line(1,0){17.5}}
\put(32,6){$<$}
\put(-5,15){\footnotesize $[1321]$}
\put(20,15){\footnotesize $[2321]$}
\put(45,15){\footnotesize $[2421]$}
\put(70,15){\footnotesize $[2431]$}
\put(95,15){\footnotesize $[2432]=\theta$}

{\color{blue}
\put(16,1){\footnotesize $1$}
\put(41,1){\footnotesize $2$}
\put(66,1){\footnotesize $3$}
\put(91,1){\footnotesize $4$}
}
\end{picture}
\caption{The subposets $\gG_\gH$ and $\widetilde{\gG_\gH}$ in $\Delta^+(\GR{F}{4})$}
\label{fig:F4}
\end{center}
\end{figure}

\begin{rmk}
For $\GR{A}{n}$, we have $\gG_\gH=\gH$, while $\gG_\gH=\{\theta\}$ for $\GR{C}{n}$. However, there
is no way to associate these subsets with extended Dynkin diagrams.
\end{rmk}
Finally, we gather the relevant numerical data.

\begin{center}
\begin{tabular}{>{$}c<{$}  | >{$}c<{$} >{$}c<{$} >{$}c<{$} >{$}c<{$} >{$}c<{$} >{$}c<{$} >{$}c<{$} >{$}c<{$} >{$}c<{$}| }
 & \GR{A}{n} & \GR{B}{n}, n\ge 3 & \GR{C}{n} & \GR{D}{n}, n\ge 4 & \GR{E}{6} & \GR{E}{7} & 
\GR{E}{8} & \GR{F}{4} & \GR{G}{2} \\[.5ex]  \hline 
\#\gG  & {n+1 \choose 2} & 2n-2 & n & 2n & 10 & 10 & 10 & 4 & 2
\\[.6ex]
\#\gG_\gH &  2n-1 & n & 1 & n+1 & 7 & 8 & 9 & 3 & 2 
\\
\end{tabular}
\end{center}
\vskip.7ex\noindent
It is curious that for the "exceptional series" $\GR{D}{5},  \GR{E}{6}, \GR{E}{7}, 
\GR{E}{8}$, the value $\#\gG$ is constant.

\end{document}